\documentclass[11pt,twoside]{article}
\usepackage{amsmath,amssymb,amsthm}

\newcommand{\Q}{\mathbb Q}
\newcommand{\Z}{\mathbb Z}
\newcommand{\F}{\mathbb F}

\providecommand{\abs}[1]{\lvert#1\rvert}

\title{THE NONEXISTENCE OF CERTAIN REPRESENTATIONS OF THE ABSOLUTE GALOIS GROUP OF QUADRATIC FIELDS}  
\author{Mehmet Haluk \c{S}eng\"{u}n}
\date{}

\begin{document}
\maketitle
\abstract{For a quadratic field $K$, we investigate continuous mod $p$ representations of $Gal(\overline{K}/K)$ that are unramified
away from $\{ p,\infty \}$. We prove that for certain $(K,p)$, there are no such irreducible representations. We also list some imaginary quadratic 
fields for which such irreducible representations exist. As an application, we  look at elliptic curves with good reduction away from $2$ over quadratic 
fields.}

\section{Introduction and Main Results}

    A famous conjecture of Serre [16] relates in a precise way the irreducible continuous odd representations of $Gal(\overline{\Q}/ \Q)$ into $GL_2(\overline{\F}_p)$
   to cuspidal modular forms over $\Q$. In particular, the conjecture implies that for $p<11$, such a representation does not exist if it is ramified
   only at $p$. 
    On the other hand, again as predicted by the conjecture, such a representation that is ramified only at $11$ does exist for $p=11$ due to the 
    elliptic curve of conductor 11.
    In support of the conjecture, Tate [18] and Serre [15] showed the nonexistence for $p=2$ and $p=3$ respectively.
    Partial nonexistence results have been proved by Brueggeman [1] for $p=5$ and by Moon and Taguchi [10] for $p=7$. Recently, the "odd level" case of
   the conjecture has been proved by Khare et al. Thus the nonexistence has been proved for $p<11$.  
   
    Let $K$ be a number field and $p$ be a rational prime. We say that the pair $(K,p)$ satisfies $(\dag)$, if there is no irreducible
    continuous representation of $Gal(\overline{K}/K)$ into $GL_2(\overline{\F}_p)$ that is unramified away from $(p, \infty )$. 
    There have been studies to formulate an analogue of Serre's conjecture over totally real fields [3] and over imaginary quadratic fields [5]. 
    A natural problem related to these studies is to know which pairs $(K,p)$ satisfy $(\dag)$ and which pairs do not.

   Let $\rho : Gal(\overline{K}/K) \rightarrow GL_2(\overline{\F}_p)$ be continuous and unramified away from $\{ p, \infty \}$. 
   Then the field $L$ corresponding to $Ker(\rho)$ is a finite extension of $K$ unramified away from $\{ p, \infty \}$ and we  
   get an embedding of $Gal(L/K)$ into some $GL_2(\F_{p^a})$. In this paper we investigate the case where $K$ is quadratic and $p=2,3$.
   
   In Section 2, we look at the case where $p=2$ and the extension $L/K$ is nonsolvable. Let $d_{K/ \Q}$ denote the discriminant of $K$ over $\Q$.

    \newtheorem*{nonsolvable}{Theorem A}
    \begin{nonsolvable}
    Let $K=\Q(\sqrt{d})$ be a quadratic field and let $L/K$ be a nonsolvable Galois extension
    unramified over every odd prime whose Galois group embeds into some $GL_2(\F_{2^a})$. 
    If $d=6,5,3,2,-1,-2,-3,-5,-6$ then no such $L$ exists.
    
    \end{nonsolvable}
    
   Brueggeman [2] proved Theorem A for $d=-2,-1,2$. 
   In Section 3, we treat the case where $p=2$ and $L/K$ is solvable for the fields reported in Theorem A.

   \newtheorem*{solvable}{Theorem B}
    \begin{solvable}
    Let $K=\Q(\sqrt{d})$ be a quadratic field and let $L/K$ be a solvable Galois extension
    unramified over every odd prime. Assume that there is an embedding $\rho:Gal(L/K) \hookrightarrow GL_2(\F_{2^a})$ for some $a$. 
    If $d=6,5,3,2,-1,-2,-3,-5,-6$ then the embedding $\rho$ is reducible.
    \end{solvable}
    
   Putting these two theorems together, we get the following result. 

   \newtheorem*{main}{Corollary}
   \begin{main}
   For $d=6,5,3,2,-1,-2,-3,-5,-6$, the pair $(\Q(\sqrt{d}),2)$ satisfies $(\dag)$.
   \end{main}
   
   In Section 4, we focus on $p=3$.
   \newtheorem*{main2}{Theorem C}
   \begin{main2}
   The pair $(\Q(\sqrt{-3}),3)$ satisfies $(\dag)$.
   \end{main2}

   We follow ideas of Tate [18] to prove the theorems.
   The proof of Theorem A is based on comparing upper and lower bounds of discriminants.
   Using a discriminant upper bound of Moon[9], one proves Theorem A
   for fields $d=5,3,2,-1,-2,-3$. To also get the fields $d=6,-5,-6$, we use part of a sharp upper bound calculation of Moon and Taguchi 
   who studied the same problem for $p=2$ in their preprint [11].   
   For Theorem B, we use class field theory and the computer algebra system MAGMA. In Section 4, we prove Theorem C by applying the 
   methods of the first two theorems to $p=3$.
   In Section 5, we present a list of imaginary quadratic fields $K$ such that $(K,2)$ does not satisfy $(\dag)$. In Section 6, we use Theorem B to
   show the nonexistence of elliptic curves with good reduction everywhere over certain quadratic fields. 
   
   \textit{Acknowledgements} \ 
   I am grateful to Nigel Boston for suggesting this problem and for his constant support throughout this project. It is a pleasure to thank 
   Yuichiro Taguchi for his important comments 
   on the preliminary version of this paper, especially on Section 2, and Seyfi T\"urkelli
   for the careful reading of the final version of this paper. Last but certainly not least, I sincerely thank 
   the referee for patiently examining the paper and for his/her very useful comments, and Ken Ono for his attention and support.

  \section{Nonsolvable Case, $p=2$}
  
   We start with the discriminant upper bound of Moon[9].
   \newtheorem*{M}{Lemma 1}
   \begin{M}Let $F$ be a finite extension of $\Q_p$ with ramification index $e$. Suppose $E/F$ is a finite extension with an elementary
   $p$-abelian Galois group of order $p^m$ where $m\geq 1$. Then the different $\mathcal{D}_{E/F}$ of $E/F$ divides $(p)^c$ where 
   \begin{displaymath}c \leq \biggl ( 1+\dfrac{\alpha}{e}\biggr )\biggl ( 1-\dfrac{1}{p^m}\biggr )\end{displaymath}
   and $\alpha=\bigl [ \frac{e}{p-1} \bigr ]+1$. (here $[x]$ denotes the maximal integer $\leq x$)
   \end{M}
   
   Observe that for $p=2$, the above upper bound takes a simple form: $c \leq \bigl (2+1/e)( 1-1/2^m)$.
   \newtheorem*{unramified}{Corollary 1}
    \begin{unramified} Let $F$ be the unramified extension of $\Q_2$. Let $E/F$ be a finite Galois extension with ramification index
    $e2^m$ with $e$ odd and $m \geq 1$. Assume that the Galois group $G$ of $E/F$ embeds into $GL_2(\F_{2^a})$ for some $a$. 
    Then the different $\mathcal{D}_{E/F}$ of $E/F$ divides $(2)^c$ where 
    
    \begin{displaymath} c \leq 3-\dfrac{1}{2^{m-1}}- \dfrac{1}{e2^m} \end{displaymath}
    
    \end{unramified}
    
    \begin{proof}
     Let $E_1$ (resp. $E_0$) be the maximal tamely ramified (resp. unramified) subextension of $E/F$. Normalize the valuation so that
     $v(2)=1$. It is well known that
     $v(\mathcal{D}_{E_1/E_0})=(e-1)/e$. As the 2-Sylow subgroups of $GL_2(\F_{2^a})$
     are elemantary 2-abelian, so is the Galois group of the extension $E/E_1$. 
     Now by Lemma 1 we have 
     \begin{displaymath} v(\mathcal{D}_{E/E_1}) \leq \biggl ( 2+\dfrac{1}{e} \biggr ) \biggl ( 1-\dfrac{1}{2^m}\biggr ) \end{displaymath}
     Combining the two differents we get
     
     \begin{center}
     \begin{tabular}{rcl} 
    $v(\mathcal{D}_{E/F}) $&$\leq$ &$\biggl ( 2+\dfrac{1}{e} \biggr ) \biggl ( 1-\dfrac{1}{2^m}\biggr )+ \biggl ( \dfrac{e-1}{e}\biggr )$ \\
    && \\
    &$\leq$&  $3-\dfrac{1}{2^{m-1}}- \dfrac{1}{e2^m}$ \\
    \end{tabular}
    \end{center}

    \end{proof}
    
    For ramified case, we will use the following upper bound calculated by Moon and Taguchi in [11].

   \newtheorem*{yeni}{Lemma 2}
    \begin{yeni}  Let $F$ be a ramified quadratic extension of $\Q_2$. Let $E/F$ be a finite Galois extension with ramification index
    $e2^m$ with $e$ odd and $m \geq 1$. Assume that the Galois group $G$ of $E/F$ embeds into $GL_2(\F_{2^a})$ for some $a$. 
    Then the different $\mathcal{D}_{E/F}$ of $E/F$ divides $(2)^c$ where 
    \begin{displaymath}c \leq \dfrac{9}{4}-\dfrac{1}{2^{m-1}} \end{displaymath}
     \end{yeni}

   \newtheorem*{discriminant}{Proposition 1}  
   \begin{discriminant} 
    Let $K$ be a quadratic number field and $L$ be a finite Galois extension of $K$ of degree $n$ which is 
    unramified over every odd prime with wild ramification index $2^m$ with $m \geq 1$.  
     Assume $Gal(L/K)$ embeds into $GL_2(\F_{2^a})$ for some $a$. Then $\abs{d_{L/\Q}} \leq \abs{d_{K/\Q}}^n 2^{2cn}$ where 
     
     \begin{itemize}
     
     \item[(a)] if $2$ is ramified in $K$, then  \ $c  \leq \dfrac{9}{4}-\dfrac{1}{2^{m-1}}$ 
     
     \item[(b)] if $2$ is  inert $K$, then \ $c \leq 3-\dfrac{1}{2^{m-1}}- \dfrac{1}{e2^m}$ 
   \end{itemize}
     
   \end{discriminant}  
     
    \begin{proof} We take a place $\mathfrak{p}$ of $K$ over $2$ and a place $\mathfrak{q}$ of $L$ over $\mathfrak{p}$. We complete $K$ and 
    $L$ at $\mathfrak{p}$ and $\mathfrak{q}$ respectively and get an extension of local fields. We apply Corollary 1 or Lemma 2 to this local 
    extension depending on the ramification of 2 in $K/ \Q$. 
    The claim follows by passing from local to global discriminant and by the fact that 
    $d_{L/\Q} = (d_{K/\Q})^{[L:K]} \   \text{Norm}_{K/\Q}(d_{L/K})$. Note that $\text{Norm}_{K/\Q}(2)=2^2$ in both cases.
    \end{proof}
    
    For lower bounds on discriminants we will use the Odlyzko-Poitou bounds [14]. Let $L/\Q$ be of degree $m$. Then
    
    \begin{displaymath} \gamma+\log{(4\pi)}-6.860404 m^{-2/3} \leq \dfrac{1}{m} \log{(\abs{d_{L/\Q}})} \end{displaymath}
    where $\gamma$ is the Euler constant.
    
    We compare these upper and lower bounds in the nonsolvable case now. Let $K$ be a quadratic field and let $L/K$ be a nonsolvable Galois extension
    ramified only over $\{ 2, \infty \}$ whose Galois group $G$ embeds into $GL_2(\F_{2^a})$ for some $a$. Let $n$ be the degree of $L/K$. 
    Note that the degree of $L/\Q$ is $2n$. 
    
    Assume that $2$ is ramified in $K/ \Q$. If $L/K$ is at most tamely ramified, then $d_{L/K}$ divides $\mathfrak{p}^n$ where $\mathfrak{p}$ 
    is a place of $K$ over $2$. Since the norm of $\mathfrak{p}$ is $2$, $\abs{d_{L/\Q}} \leq \abs{d_{K/\Q}}^n  2^n$ . Thus
    
    \begin{displaymath} 2(\gamma+\log{(4\pi)}-6.860404(2n)^{-2/3}) \leq  \log{\abs{d_{K/Q}}} +  \log{2} \end{displaymath}
    As $G$ is nonsolvable, $n \geq 60$. For $\abs{d_{K/\Q}} \leq 2^{128}$, this inequality gives a contradiction for all $n\geq 60$.
    
    Now assume that $L/K$ is wildly ramified with ramification index $2^m$. Using Lemma 2, we have
    \begin{displaymath} 2(\gamma+\log{(4\pi)}-6.860404(2n)^{-2/3}) \leq \log{\abs{d_{K/Q}}}+ 2c\log{2} \end{displaymath}
    where $c=\frac{9}{4}-\frac{1}{2^{m-1}}$.
    
    As Tate observes in [18], we have $\frac{n}{2^{m-1}}\geq 30$ because $2^m$ divides $n$ and $n$ is divisible by at least three distinct primes
    as it is the order of a nonsolvable group. Now we get
    
    \begin{displaymath} 2 \biggl ( .5772+2.53102-\dfrac{6.860404}{2^{2/3}n^{2/3}} \biggr ) \leq \log{\abs{d_{K/Q}}}+ 1.386295 \biggl (2.125-\dfrac{30}{n} \biggr )\end{displaymath} 
    
    \begin{displaymath} 6.216448-\dfrac{8.64356}{n^{2/3}} \leq \log{\abs{d_{K/Q}}} + 2.94587-\dfrac{41.588}{n}\end{displaymath} 
    
   \begin{displaymath} 3.27057+f(n) \leq \log{\abs{d_{K/Q}}}  \end{displaymath} 
    where $f(x)=\frac{A-Bx^{1/3}}{x}$ with $A=41.588$ and $B=8.64356$. The function $f(x)$ decreases until it reaches its minimum 
    at $x_0=(\frac{3A}{2B})^3 \approx 375.923$ with minimum value $f_{min}= \frac{-A}{2x_0}$ and then
    it increases approaching $0$ as $x$ tends to infinity. So, if  $\log{\abs{d_{K/Q}}} \leq 3.27057 + f_{min}\approx  3.21525$, 
   the last inequality gives a contradiction for any $n \geq 60$. Thus we get $\abs{d_{K/Q}} < 24.9$,  
   proving the claim for the fields $K=\Q (\sqrt{d})$ with $d=6,3,2,-1,-2,-5,-6$.
    
    Now assume that $2$ is inert in $K/ \Q$. As $e2^m$ is the order of the solvable local inertia group, its index in nonsolvable $G$ 
    has to be at least 3, thus $\frac{n}{e2^m} \geq 3$. Using Corollary 1, we get
    
    \begin{displaymath} 6.216448-\dfrac{8.64356}{n^{2/3}} \leq \log{\abs{d_{K/Q}}} + 1.386295 \biggl (3-\dfrac{30}{n}- \dfrac{3}{n} \biggl )\end{displaymath}

    \begin{displaymath} 2.057563+g(n) \leq \log{\abs{d_{K/Q}}} \end{displaymath} 
    where $g(x)=\frac{A-Bx^{1/3}}{x}$ with $A=45.7477$ and $B=8.64356$. The minimum value of $g(x)$ is attained at $x_0 \approx 500.385$. If 
   $\log{\abs{d_{K/Q}}} \leq 2.057563 + g_{min}\approx  2.011863$, 
   the last inequality gives a contradiction for any $n \geq 60$. 
    Thus we get $\abs{d_{K/Q}} < 7.477$, proving the claim for the fields $K=\Q (\sqrt{d})$ with $d=-3,5$.
    
   This completes the proof Theorem A.

    \vspace{.1 in}

    \section{Solvable Case, $p=2$}

    Let $L/K$ be a solvable Galois extension with Galois
    group $G$ that is ramified only over $\{ 2, \infty \}$. Assume that there is an embedding $\rho : G \hookrightarrow GL_2(\F_{2^a})$ for some $a$.
    If we show that $G$ is a $2$-group then a conjugate of the image of $G$ will be inside the Sylow 2-subgroup
   $ T=\lbrace \bigl( \begin{smallmatrix} 1 & x  \\ 0 & 1 \\ \end{smallmatrix} \bigr) | \ x \in \F_{2^a} \rbrace$ of $GL_2(\F_{2^a})$. 
   Thus $\rho$ will be reducible.

    Let $S$ be a $2$-Sylow subgroup of $G$. Then $S$ is elementary $2$-abelian as $T$ is.
    Let $G'$ be the commutator subgroup of $G$. To show that $G$ is a $2$-group, it is enough to show that $G/G'$  and  $G'/G''$ are $2$-groups. 
    If they are, then $G/G''$ is a 2-group and it is abelian as it is a homomorphic image of $S$. Indeed, 
   $G/G'' \backsimeq SG''/G'' = S/S \cap G''$. Hence $G'=G''$. Since $G$ is solvable, 
    we have $G'=1$ and thus $G$ is a $2$-group. 
    
    \vspace{.1 in}

    \textit{In the rest of this section, $K=\Q(\sqrt{d})$ with $d=6,5,3,2,-1,-2,-3,-5$ or $-6$.} 
    
    \vspace{.1 in}

    Observe that 2 is either inert ($d= -3, 5$) of ramified in $K/ \Q$. Let $\mathfrak{p}$ denote the only place of $K$ above 2. 
    We will prove that $G/G'$  and  $G'/G''$ are $2$-groups.

   \newtheorem*{rayclass}{Proposition 2}
    \begin{rayclass}
    The ray class group of $K$ with modulus $\mathfrak{p}^k \mathfrak{m}_{\infty}$ is a $2$-group for any $k$ where $\mathfrak{m}_{\infty}$
    is the modulus of all the real archimedean places of $K$.
   
    \end{rayclass}
    
    \begin{proof}
    Let $\mathcal{O}_K$ be the ring of integers of $K$ and $U$ be the group of units of $\mathcal{O}_K$. 
    Let $ \operatorname{Cl}(K)$ be the ideal class group of $ K$ and let $ \operatorname{Cl}(K,\mathfrak{p}^k\mathfrak{m}_{\infty})$ be 
    the ray class group of $K$ of modulus $ \mathfrak{p}^k\mathfrak{m}_{\infty}$ with fixed positive integer $k$.
    
    We have the following exact sequence from class field theory
    
    $$ (\ast) \ \ \ U  \rightarrow (\mathcal{O}_K/\mathfrak{p}^k)^* \times |\Z / 2\Z|^{|\mathfrak{m}_{\infty}|}\rightarrow \operatorname{Cl}(K,\mathfrak{p}^k\mathfrak{m}_{\infty}) \rightarrow  \operatorname{Cl}(K) \rightarrow 1$$
    
    It is known that the prime to $2$ part of $(\mathcal{O}_K/\mathfrak{p}^k)^*$ is $\Z / (2^f-1)\Z$ where $f$ is the residue degree of $\mathfrak{p}$. 
    Thus if $2$ is ramified in $K$, then $(\mathcal{O}_K/\mathfrak{p}^k)^*$ is a $2$-group. Since the class numbers of $K$'s are all powers of 2, the result
    follows in this case. If $2$ is inert, there may be a non-trivial $3$-part of the ray class group. Note that the $3$-rank is the same for every $k$. For the two
   inert fields, we verify with MAGMA that the ray class group with modulus $(2)\mathfrak{m}_{\infty}$ has $3$-rank zero for all $d$'s.
    \end{proof}
    
   Let $F$ be the fixed field of $G'$. Then $F$ is an abelian extension of $K$ that is ramified only over $\{ 2,\infty \}$ and $F$ is contained in a ray class field
   of $K$ with modulus $\mathfrak{p}^k \mathfrak{m}_{\infty}$ for some $k$. By Proposition 2, such a ray class field has degree power of $2$ over $K$. 
   Thus $G/G'$ is a $2$-group. 
   
   The group $G'/G''$ corresponds to an abelian extension of $F$ that is only ramified over $\{ 2,\infty \}$ and thus is contained in a ray class field
   of $F$ with modulus $(2)^k \mathfrak{m}_{\infty}$ for some $k$.
   Using MAGMA, we will verify for each possible $F$ that these ray class groups are $2$-groups and conclude that $G'/G''$ is a 2-group.
   First, we use the following theorem of Nakagoshi [12] to find a field $A$
   which contains all possible $F$'s.
   
   \newtheorem*{Nakagoshi}{Theorem 1}
   \begin{Nakagoshi}
   Let $N$ be a number field with ramification index $e$ and residue degree $f$ over the rational prime $p$ and let 
   $\mathfrak{p}$ be a prime ideal of the ring of integers $\mathcal{O}$ of $N$ over $p$. Set 
   $e_1=\biggl [ \dfrac{e}{p-1} \biggr ]$ where $[x]$ is the maximal integer $\leq x$. Let $N_{\mathfrak{p}}$ denote the completion
   of $N$ at $\mathfrak{p}$. Then the $p$-rank $R_n$ of $(\mathcal{O}/\mathfrak{p}^{n+1})^*$ is given by

   \begin{center}
   \begin{tabular}{lcl}
   $R_n= \biggl ( n- \biggl [ \dfrac{n}{p} \biggr ] \biggr ) f,$&&if $0 \leq n < e+e_1$ \\
   $R_n=ef$,&& if $n \geq e+e_1$ and $\zeta_p \not \in N_{\mathfrak{p}}$\\
   $R_n=ef+1$,&& if $n \geq e+e_1$ and $\zeta_p  \in N_{\mathfrak{p}}$ \\
   \end{tabular}
   \end{center} 
   \end{Nakagoshi}   

   Combining this result with the exact sequence $(\ast)$, we see that the $2$-ranks of ray class groups of modulus $(2)^k\mathfrak{m}_{\infty}$ stabilize
   after $k=5$ for every quadratic field. Thus there exists a maximal elementary 2-abelian extension $A$ of $K$ that is only ramified over $\{ 2, \infty \}$.
   As $G/G'$ is elementary $2$-abelian (it is a homomorphic image of $S$), $F$ is a subfield of $A$. 
   For every $d$, we list a defining polynomial of $A$ over $\Q$, class number $h$ of $A$ and the decomposition ($e,f,g$) of $2$ in $A/\Q$.
   
   \vspace{.1 in}
   
   \begin{center}
   \begin{tabular}{|c|c|c|c|} \hline
$\textbf{d}$ & $\textbf{A}$ & $\textbf{h}$ & $\textbf{e,f,g}$ \\ \hline\hline

6& $x^{16} + 4x^{12} + 15x^8 + 4x^4 + 1$ & 1 & 8,2,1 \\ \hline
5& $x^{16} - 12x^{14} + 58x^{12} - 29x^8 + 58x^4 + 12x^2 + 1$ & 1 & 8,2,1 \\ \hline
3& $x^{16} + 4x^{14} + 56x^{12} + 36x^{10} + 542x^8 + 636x^6 + 248x^4 + 28x^2 + 1$ & 1 & 8,2,1 \\ \hline
2& $x^{16} + 4x^{12} + 40x^{10} + 104x^8+112x^6 + 56x^4 + 16x^2 + 4 $& 1 & 16,1,1 \\ \hline
$-1$ & $x^8 + 4x^6 + 22x^4 + 4x^2 + 1$ &  1 & 8,1,1 \\ \hline
$-2$ & $x^8 + 4x^6 + 10x^4 - 20x^2 + 9$ & 1 & 8,1,1 \\ \hline
$-3$ & $x^8 - 10x^6 + 31x^4 - 6x^2 + 9$ & 1 & 4,2,1 \\ \hline
$-5$ & $x^8 + 32x^6 + 248x^4 + 512x^2 +16$ & 1 & 4,2,1 \\ \hline
$-6$ & $x^8 + 24x^6 + 248x^4 - 288x^2 +2704$ & 1 & 4,2,1 \\ \hline

\multicolumn{4}{c}{\begin{small} Table 1 \end{small}} \\
\end{tabular} 
\end{center}   

\vspace {.1 in}

    We compute the class numbers of all subfields of $A$ for every $d$ and see that they are all powers of $2$.

    From Table 1, we see that residue degree $f$ of $2$ in $A$ is either 1 or 2. We also observe that each subfield of $A$ has only one place
    over $2$. By the exact sequence $(\ast)$, we see that for the subfields of $A$ 
    with $f=1$, the 3-rank of its ray class group with modulus $(2)^k\mathfrak{m}_{\infty}$ will be 0.
    For the subfields of $A$ with $f=2$ which contain $K$, 
    we check the $3$-rank of their ray class groups with modulus $(2)\mathfrak{m}_{\infty}$ and see that it is $0$ in all instances. 
    This shows that $G'/G''$ is a $2$-group for all the quadratic fields listed in Theorem B and thus completes the proof of Theorem B. 
    
    \section{The Case $p=3$}
    We apply Lemma 1 to the case $p=3$ and get the following: 
    
    \newtheorem*{discriminant3}{Proposition 3}  
   \begin{discriminant3} 
    Let $K$ be a quadratic field ramified over $3$ and $L$ be a finite Galois extension of $K$ of degree $n$ which is 
    unramified away from $\{ 3,\infty \}$. Let the ramification index of $L/K$ be $e3^m$ with $m \geq 1$ and $(e,3)=1$.  
    Assume $Gal(L/K)$ embeds into some $GL_2(\F_{3^a})$. Then $\abs{d_{L/\Q}} \leq \abs{d_{K/\Q}}^n 3^{2cn}$ where 
     
    \begin{displaymath} c  \leq 2-\dfrac{1}{2\cdot 3^{m-1}}-\dfrac{1}{2 e \cdot 3^m}\end{displaymath}

   \end{discriminant3}  
  
   \begin{proof} Just as in the proof of Proposition 1, we look at the local differents. 
    We suitably complete $K$ and $L$ over $3$ to get the local extension $E/F$. Let $E_1$ (resp. $E_0$) be the maximal tamely ramified 
    (resp. unramified) subextension of $E/F$. Normalize the valuation so that $v(3)=1$. We have
     $v(\mathcal{D}_{E_1/E_0})=(e-1)/2e$. Gal$(E/E_1)$ is an elementary 3-abelian group and by
    Lemma 1, we see that $\mathcal{D}_{E/F}$ divides $(3)^c$ where 
   
   \begin{center}
   \begin{tabular}{rl}  
   $c$  & $\leq\biggl ( 1+\dfrac{\alpha}{2e}\biggr )\biggl ( 1-\dfrac{1}{3^m}\biggr )+ \dfrac{e-1}{2e}$ \\ 
   & \\
   & $ \leq \biggl ( 1+\dfrac{1}{2}+ \dfrac{1}{2e} \biggr )\biggl ( 1-\dfrac{1}{3^m}\biggr ) +\dfrac{1}{2}- \dfrac{1}{2e}$ \\
   &\\
   & $ \leq \dfrac{3}{2}+\dfrac{1}{2}-\dfrac{1}{2 \cdot 3^{m-1}} -\dfrac{1}{ 2e \cdot 3^m}$ \\

  \end{tabular}
  \end{center}
  We pass to the local and then to global discriminant and get the desired result.
   \end{proof}
  
    We follow Section 2 and Section 3 to prove Theorem C. Let $L/K$ be an extension satisfying the hypothesis of Proposition 3. 
    Assume that $L/K$ is nonsolvable. Using the lower bound of Section 2, we get
      
    \begin{displaymath} 6.216448-\dfrac{8.64356}{n^{2/3}} \leq \log{\abs{d_{K/\Q}}}+ 2.197225 \biggl (2-\dfrac{33}{2n} \biggr )\end{displaymath}

    \begin{displaymath} 1.82198+h(n) \leq \log{\abs{d_{K/\Q}}}  \end{displaymath}
    where $h(x)=\frac{A-Bx^{1/3}}{x}$ with $A=41.36.254$ and $B=8.64356$. The minimum is attained at $x_0 \approx 249.041$. For 
    $\log{\abs{d_{K/Q}}} < 1.82198 + h_{min}\approx  1.749$, 
   the last inequality gives a contradiction for any $n \geq 60$. Thus we get $\abs{d_{K/ \Q}} < 5.7 $,  
   only giving $\Q(\sqrt{-3})$.
   
    Now let $L/\Q(\sqrt{-3})$ be a solvable extension satisfying the hypothesis of Proposition 3. Let $G$ be the Galois group of
    this extension. We want to show that $G/G'$ and $G'/G''$ are both 3-groups following Section 3. By the exact sequence $(\ast)$ of Section 3,
    we see that any ray class group of $\Q(\sqrt{-3})$ with modulus $(3)^k \mathfrak{m}_{\infty}$ is a $3$-group because 
    the class number of $\Q(\sqrt{-3})$ is 1 and it has no infinite places and residue degree of 3 is one. Thus $G/G'$ is a $3$-group.
    Now let $A$ be the maximal elementary 3-abelian extension of $\Q(\sqrt{-3})$ that is unramified over $\{ 3, \infty \}$. Using MAGMA, we find
    a defining polynomial for $A$ over $\Q$ : $x^{18} - 9x^{15} + 135x^{12} + 540x^9 + 2673x^6 + 1458x^3 + 729$. The decomposition of $3$ in $A/ \Q$
    is $(18,1,1)$ which means for any subfield the residue degree of 3 is one as well. We verify that all subfields of $A$ 
    containing $K$ have class number 1 and have no real infinite places. As in Section 3, we conclude that $G'/G''$ is a 3-group. This proves Theorem C.
    
    \section{A List of Pairs Not Satisfying $(\dag)$}
    We now investigate the pairs $(K,2)$ for which $(\dag)$ fails. 
    The simplest case is a $GL_2(\F_2) \backsimeq Sym(3)$ extension $L/K$ that is ramified only over $\{2, \infty \}$. Using group theory with MAGMA, we have searched 
   the number fields database of J.Kl\"{u}ners and  G.Malle [8] for $Sym(3)$ extensions of quadratic fields with little or no ramification. 
   In Table 2, we list some of our findings for imaginary $K$. In each case, $L$ is the splitting field of the given polynomial over $\Q$ and the third column is ramification
   of finite places in $L/K$.

   \begin{center}
   \begin{tabular}{|c|c|c|} \hline
   $d$ & $f(x)$& ramification \\ \hline
   $-13$&$ x^6 + x^4 + 4x^3 + 36x^2 - 24x + 4$& only over $2$ \\ \hline
   $-19$&$x^6 - 8x^5 + 23x^4 - 24x^3 + x^2 + 14x + 4$& only over $2$ \\  \hline
   $-22$&$x^6 - 2x^5 + 5x^4 + 8x^3 + 47x^2 + 90x + 47$& only over $2$ \\  \hline
   $-37$&$x^6 + 4x^5 + 23x^4 - 4x^3 + 71x^2 - 288x + 293$& only over $2$ \\  \hline
   $-38$&$x^6 + 6x^5 + 33x^4 + 60x^3 + 89x^2 - 258x + 207$& only over $2$ \\  \hline
   $-46$&$x^6 + 6x^5 + 21x^4 + 52x^3 + 291x^2 + 326x + 271$& unramified \\ \hline
   $-58$&$x^6 + 8x^5 + 40x^4 + 60x^3 +261x^2 + 380x + 382$ & only over $2$ \\  \hline
   $-62$&$ x^6 + 6x^5 + 45x^4 + 132x^3 + 179x^2 +246x + 423$& unramified \\  \hline
   $-74$&$x^6 + 6x^5 + 41x^4 + 32x^3 +101x^2 - 654x + 691$& only over $2$ \\  \hline
   $-79$&$ x^6 - 3x^5 + 14x^4 - 4x^3 + 40x^2+ 64x + 64$ & only over $2$ \\ \hline
      
   \multicolumn{3}{c}{Table 2}
   
   \end{tabular}
   \end{center}
  
   \section{Application to Elliptic Curves over Quadratic Fields} 
   
   Let $K=\Q(\sqrt{d})$ be a quadratic field. Assume that $E$ is an elliptic curve over $K$ that has good reduction away from $2$. 
   Let $G$ be the Galois group of the finite extension
   $K(E[2])/K$ where $K(E[2])$ is the extension of $K$ obtained by adjoining coordinates of points of $E$ that are of order $2$.
   It is well known that there is a continuous representation $$\rho: G \hookrightarrow GL_2(\F_2)$$
   which is ramified away from $2$. If $d=6,5,3,2,-1,-2,-3,-5,-6$ then by the proof of Theorem B, $G$ must be a $2$-group. This implies that $G$ is either trivial or
    it is $\Z / 2$. This is true only if $E$ has a $K$-rational point of order 2. Thus we showed that
    
   \newtheorem*{elliptic}{Proposition 4}
   \begin{elliptic}
   For $d=6,5,3,2,-1,-2,-3,-5,-6$, 
   if $E$ is an elliptic curve over $K$ that has good  reduction away from $2$ then $E$ has a $K$-rational point of order $2$.
   \end{elliptic}

   This extends results of Pinch[13] and Kida[7].

   An elliptic curve $E$ over $K$ is called $\textit{admissible}$ if the following conditions are 
   satisified:
   
   \begin{itemize}
   \item[(1)] $E$ has good reduction everywhere over $K$
   \item[(2)] $E$ has a $K$-rational point of order $2$
   \end{itemize}
    
   Comalada[4] showed that for $1<d<100$, there exists an admissible elliptic curve over $\Q(\sqrt{d})$ if and only if $d=6,7,14,22,38,41,65,77,86$.
   Setzer[17] showed that for $d<0$, there exists an admissible elliptic curve over $\Q(\sqrt{d})$ if and only if $d=65d_1$ where $d_1$ is a square
   modulo $5$ and modulo $13$ and $65$ is a square modulo $d_1$. Combining these two results with Proposition 4, we get
   
   \newtheorem*{ellcorollary}{Corollary 3}
   \begin{ellcorollary} 
   For $d=5,3,2,-1, -2, -3, -5, -6$, there is no elliptic curve with good reduction everywhere over $\Q(\sqrt{d})$.
   \end{ellcorollary}
   
   Kagawa and Kida proved the nonexistence of elliptic curves with good reduction everywhere over many small quadratic fields , 
   including the ones listed in this corollary (see [6] and Kagawa's thesis).
   One may try to use our approach on the several other small ones quadratic fields not covered by their methods. 
   
   \section{References}
   
   \begin{small}
   \begin{itemize}
   \item[1.] S.Brueggeman ; The nonexistence of certain nonsolvable Galois extensions unramified outside $5$. {\it J. Number Theory} {\bf 75} (1999), 47--51
   \item[2.] S.Brueggeman ; The nonexistence of certain nonsolvable Galois extensions of number fields of small degree.  {\it Int. J. Number Theory}  {\bf 1}  (2005),  no. 1, 155--160
   \item[3.] K.Buzzard, F.Diamond, A.F.Jarvis ;  On Serre's conjecture for mod l Galois representations over totally real fields. preprint 
   \item[4.] S.Comalada ; Elliptic curves with trivial conduxtor over quadratic fields. {\it Pacific J.Math}, {\bf 144}, (1990), 237-258
   \item[5.] L.M.Figueiredo ; Serre's conjecture for imaginary quadratic fields. {\it Compositio Math.} {\bf 118} (1999), no. 1, 103--122
   \item[6.] T.Kagawa, M.Kida ; Nonexistence of elliptic curves with good reduction everywhere over real quadratic fields. {\it J. Number Theory} {\bf 66} (1997), 201--210
   \item[7.] M.Kida ; Reduction of elliptic curves over certain real quadratic number fields. {\it Math. Comp.}, {\bf 68}  (1999),  no. 228, 1679--1685
   \item[8.] J.Kl\"{u}ners, G.Malle; {\footnotesize http://www.math.uni-duesseldorf.de/$\sim$klueners/minimum/minimum.html}
   \item[9.] H.Moon ; Finiteness results on certain mod p Galois representations, {\it J. Number Theory} {\bf 84} (2000), 156--165 
   \item[10.] H.Moon, Y.Taguchi ; Refinement of Tate's discriminant bound and nonexistence theorems for mod $p$ Galois representations. {\it Doc.Math.}, Extra vol. (2003), 641--654
   \item[11.] H.Moon, Y.Taguchi ; The nonexistence of certain mod 2 Galois representations of some small quadratic fields, preprint, http://webbuild.knu.ac.kr/~hsmoon/bib/mod2qf.htm
   \item[12.] N.Nakagoshi ; The structure of the multiplicative group of residue classes modulo $\mathfrak{p}^{N+1}$.  {\it Nagoya Math. J.}  {\bf 73}  (1979), 41--60.
   \item[13.] R.G.E.Pinch; Elliptic curves with good reduction away from $2$. II. {\it Math. Proc. Cambridge Philos. Soc.}  {\bf 100}  (1986),  no. 3, 435--457. 
   \item[14.] G.Poitou (d'apr\`es A. M. Odlyzko); Minorations de discriminants. \textit{S\'em. Bourbaki}, 1975/76, Exp.479,  pp. 136--153. Lecture Notes in Math., Vol. \textbf{567}, Springer, Berlin, 1977.
   \item[15.] J.P.Serre ; Oeuvres. Vol.III, p.710, Springer-Verlag, Berlin, 1986 
   \item[16.] J.P.Serre ; Sur les representations modulaires de degre 2 de $Gal(\overline{\Q}/ \Q)$. {\it Duke Math.J.}, {\bf 54}, (1987)
   \item[17.] B.Setzer ; Elliptic curves over complex quadratic fields. {\it Pacific J. Math.}  {\bf 74}  (1978), no. 1, 235--250
   \item[18.] J.Tate ; The non-existence of certain Galois extensions of $\Q$ unramified outside $2$.  Arithmetic geometry (Tempe, AZ, 1993),  153--156, 
    Contemp. Math., 174, Amer. Math. Soc., Providence, RI, 1994  
   
   \vspace{.2 in}
   \textsc{department of mathematics, university of wisconsin, 480 Lincoln Dr Madison wi 53706}
   
   \textit{E-mail address:} sengun@math.wisc.edu

   \end{itemize}
   \end{small}

\end{document}